\documentclass[12pt,reqno,twoside]{article}
\usepackage{amsfonts,amssymb,amsmath,amsthm}
\usepackage[top=20mm,bottom=20mm,left=20mm,right=20mm]{geometry}
\usepackage[T1]{fontenc}
\usepackage{times}
\usepackage{multicol}
\usepackage{amsfonts}
\usepackage{graphicx}
\usepackage{graphics}
\usepackage{enumitem}
\usepackage{cite}
\usepackage{fancyhdr}
\fancypagestyle{titlepage}
{
   \fancyhf{}
   \lhead{The final publication is available at "http://dergiler.ankara.edu.tr/dergiler/29/2239/23277.pdf"}
   \rhead{}
   \fancyfoot[C]{\thepage}
}

\newcommand\blfootnote[1]{%
  \begingroup
  \renewcommand\thefootnote{}\footnote{#1}%
  \addtocounter{footnote}{-1}%
  \endgroup
}

\newenvironment{rcases}
  {\left.\begin{aligned}}
  {\end{aligned}\right\rbrace}
\newtheorem{theorem}{Theorem}[section]
\numberwithin{equation}{section}
\newtheorem{definition}[theorem]{Definition}

\newtheorem{corollary}[theorem]{Corollary} %%Delete [teo] to re-start numbering
\newtheorem{lemma}[theorem]{Lemma} %%Delete [teo] to re-start numbering
\newtheorem{proposition}[theorem]{Proposition}
\newtheorem{remark}[theorem]{Remark}
\theoremstyle{definition}
\newtheorem{example}[theorem]{Example}
\begin{document}

\title{Ces\`{a}ro summability of integrals of fuzzy-number-valued functions}
\author{Enes Yavuz\footnote{Corresponding author. E-mail: enes.yavuz@cbu.edu.tr }\ \ , \"{O}zer Talo and  H\"{u}samettin \c{C}o\c{s}kun }

\date{{\small Department of Mathematics, Manisa Celal Bayar University, Manisa, Turkey.\\ E-mails: enes.yavuz@cbu.edu.tr; ozer.talo@cbu.edu.tr; husamettin.coskun@cbu.edu.tr }}

\maketitle
\thispagestyle{titlepage}
\blfootnote{\rule{0.63cm}{0cm}\emph{Mathematics Subject Classification:} 40G05, 40E05, 40A10, 03E72}

\begin{abstract}
In the present study, we have introduced Ces\`{a}ro summability of integrals of fuzzy-number-valued functions and given one-sided Tauberian conditions under which convergence of improper fuzzy Riemann integrals follows from Ces\`{a}ro summability. Also, fuzzy analogues of Schmidt type slow decrease and Landau type one-sided Tauberian conditions have been obtained.
\end{abstract}

\section{Introduction}
Given a locally integrable function $f:[0,\infty)\rightarrow \mathbb{C}$, {\it the Ces\`{a}ro operator Cf} is defined by
$$(Cf)(x):=\frac{1}{x}\int_0^{x}f(t)dt, \qquad x\in(0,\infty).$$
In classical analysis, the Ces\`{a}ro operator was investigated from various aspects and a large number of results have appeared recently\cite{operator1,operator2,operator3,operator4,operator5}. Titchmarsh\cite{titchmarsh} also used the operator as a convergence method for divergent integrals and introduced the Ces\`{a}ro summability of integrals\cite[p.11]{hardy}. Following this introduction, the concept of Ces\`{a}ro summability of integrals received considerable attention and Tauberian conditions under which Ces\`{a}ro summable improper integrals converge have been investigated\cite{hardy,TT,moric1,new1,new2,new3,new4,new5,new6}. Also, there are studies applying the concept to Fourier integrals\cite{titchmarsh,fourier1,fourier2,fourier3,fourier4}.

In the light of the developments mentioned above, establishment of the concept of Ces\`{a}ro summability of integrals for fuzzy analysis is also of
 importance for handling divergent integrals of fuzzy-number-valued functions. The concept of integration of fuzzy-number-valued functions has already been  introduced by  Dubois et al. \cite{fek1} and  studied by many mathematicians\cite{fuzzy1,an,fuzzy2,fuzzy3}. Also, in particular, Bede and Gal\cite{bede} have proved that there exists a mean value, or a Ces\`{a}ro sum,  for any almost periodic fuzzy-number-valued function and given some applications of these functions to fuzzy differential equations and to fuzzy dynamical systems. At this point, approaching the concept of \lq mean value\rq\ from perspective of summability theory, we define Ces\`{a}ro summability of integrals of fuzzy-number-valued functions and give various types of convergence conditions  for Ces\`{a}ro summable improper integrals of fuzzy-number-valued functions.
\section{ Preliminaries}

A \textit{fuzzy number} is a fuzzy set on the real axis, i.e. u is normal, fuzzy convex, upper semi-continuous and $\operatorname{supp}u =\overline{\{t\in\mathbb{R}:u(t)>0\}}$ is compact \cite{zadeh}.
We denote the space of fuzzy numbers by $E^1$. \textit{$\alpha$-level set} $[u]_{\alpha}$ of $u\in E^1$
is defined by
\begin{eqnarray*}
[u]_{\alpha}:=\left\{\begin{array}{ccc}
\{t\in\mathbb{R}:u(t)\geq\alpha\} & , & \qquad if \quad 0<\alpha\leq 1, \\[6 pt]
\overline{\{t\in\mathbb{R}:u(t)>\alpha\}} & , &
if \quad \alpha=0.\end{array} \right.
\end{eqnarray*}
Each $r\in\mathbb{R}$ can be regarded as a fuzzy number $\overline{r}$ defined by
\begin{eqnarray*}
\overline{r}(t):=\left\{\begin{array}{ccc}
1 & , & if \quad t=r, \\
0 & , & if \quad t\neq r.\end{array} \right.
\end{eqnarray*}
Let $u,v\in E^1$ and $k\in\mathbb{R}$. The addition and scalar multiplication are defined by
\begin{eqnarray*}
[u+v]_{\alpha}=[u]_{\alpha} +[v]_{\alpha}=[u^-_{\alpha}+v^-_{\alpha}, u^+_{\alpha}+v^+_{\alpha}], [k u]_{\alpha}=k[u]_{\alpha}
\end{eqnarray*}
where $[u]_{\alpha}=[u^-_{\alpha}, u^+_{\alpha}]$, for all $\alpha\in[0,1]$.

\begin{lemma}\label{added}\cite{bede} The following statements hold:
\begin{itemize}
\item[(i)] $\overline{0} \in E^{1}$ is neutral element with respect to $+$, i.e., $u+\overline{0}=\overline{0}+u=u$ for all $u \in E^{1}$.
\item [(ii)] With respect to $\overline{0}$, none of $u \neq \overline{r}$, $r\in \mathbb{R}$
has opposite in $E^{1}.$
\item[(iii)] For any $a,b \in\mathbb{R}$ with $a, b \geq 0$ or $a,b \leq 0$ and any $u\in E^{1}$, we have
$(a + b) u = au + bu$. For general $a, b \in \mathbb{R}$, the above property does not hold.
\item[(iv)] For any $a \in \mathbb{R}$ and any $u, v \in E^{1}$, we have
$a (u+ v) = au + av.$
\item[(v)] For any $a, b \in \mathbb{R}$ and any $u \in E^{1}$, we have
$a (b u) = (a b) u.$
\end{itemize}
\end{lemma}
The metric $D$ on $E^1$ is defined as
\begin{eqnarray*}
 D(u,v):=\sup_{\alpha\in[0,1]}d([u]_{\alpha},[v]_{\alpha}):=
\sup_{\alpha\in[0,1]}\max\{|u^-_{\alpha}-v^-_{\alpha}|,|u^+_{\alpha}-
v^+_{\alpha}|\}.
\end{eqnarray*}
where  $d$ is the Hausdorff metric.
\begin{proposition}\cite{bede}
\label{p02} Let $u,v,w,z\in E^1$ and $k\in\mathbb{R}$. Then,
\begin{itemize}
\item [(i)] $(E^1,D)$ is a complete metric space.
\item [(ii)] $D(ku,kv)=|k|D(u,v)$.
\item [(iii)] $D(u+v,w+v)=D(u,w)$.
\item [(iv)] $D(u+v,w+z)\leq D(u,w)+D(v,z)$.
\item [(v)] $|D(u,\overline{0})-D(v,\overline{0})|\leq D(u,v)\leq
D(u,\overline{0})+D(v,\overline{0})$.
\end{itemize}
\end{proposition}
Partial ordering relation on $E^1$ is defined as follows:
\begin{eqnarray*}
u\preceq v \Longleftrightarrow
[u]_{\alpha}\preceq[v]_{\alpha}\Longleftrightarrow u^-_{\alpha}\leq
v^-_{\alpha}~\text{ and }~u^+_{\alpha}\leq v^+_{\alpha}~\text{ for
all }~\alpha\in[0,1].
\end{eqnarray*}

We say a fuzzy number $u$ is negative if and only if $u(t)=0$ for all $t\geq 0$ (see \cite{negative}).\\
Combining the results of Lemma 6 in \cite{aytarfss1}, Lemma 5 in \cite{aytarinf3}, Lemma 3.4, Theorem 4.9 in \cite{Li}  and Lemma 14 in\cite{slowly}, following Lemma is obtained.
\begin{lemma}\label{epsilon}
Let $u,v,w,e\in E^1$ and $\varepsilon>0$. The following statements hold:
\begin{itemize}
  \item [(i)] $D(u,v)\leq\varepsilon$ if and only if $u-\overline{\varepsilon}\preceq v \preceq u+\overline{\varepsilon}$
  \item [(ii)] If $u\preceq v +\overline{\varepsilon}$ for every $\varepsilon>0$,  then $u\preceq v$.
  \item [(iii)] If $u\preceq v$ and $v\preceq w$, then $u\preceq w$.
   \item [(iv)] If  $u\preceq w$ and $v\preceq e$, then $u+v\preceq w+e$.
    \item [(v)] If  $u+w\preceq v+ w$ then $u\preceq v$.
\end{itemize}
\end{lemma}
\begin{definition}
A fuzzy-number-valued function $f:[a,b]\to E^1$ is said to be continuous at $x_0\in [a,b]$ if for each $\varepsilon>0$ there is a $\delta>0$ such that $D(f(x), f(x_0))<\varepsilon$ whenever $x\in [a,b]$ with $|x-x_0|<\delta$. If $f(x)$ is continuous at each $x\in [a,b]$, then we say $f(x)$ is continuous on $[a,b]$.
\end{definition}
\begin{definition}\cite{gv}
A fuzzy-valued function $f:[a,b]\to E^1$ is called Riemann integrable on $[a,b]$, if there exists $I\in E^1$ with the property : $\forall\varepsilon>0$, $\exists\delta>0$ such that for any division of $[a,b]$ $d : a=x_0<x_1<\cdots<x_n=b$ of norm $v(d)<\delta$, and for any points $\xi_i\in[x_i,x_{i+1}]\ i=\overline{0,n-1}$, we have $$D\left(\sum\limits_{i=0}^{n-1}f(\xi_i)(x_{i+1}-x_i), I\right)<\varepsilon.$$ Then $I=\int\limits_a^bf(x)dx$.
\end{definition}
\begin{theorem}\cite{gv}\label{fint}
If the fuzzy-number-valued function $f:[a,b]\to E^1$ is continuous (with respect to the metric D) and for each $x\in[a,b]$, $f(x)$ has the parametric representation $$[f(x)]_{\alpha}=[f_{\alpha}^-(x), f_{\alpha}^+(x)],$$ then $\int\limits_a^bf(x)dx$ exists, belongs to $E^1$ and is parametrized by
$$\left[\int\limits_a^bf(x)dx\right]_{\alpha}=\left[\int\limits_a^bf_{\alpha}^-(x)dx, \int\limits_a^bf_{\alpha}^+(x)dx\right].$$
\end{theorem}
Using the results of Anastassiou \cite{an} we have
\begin{theorem}\label{bessart}
If $f:[a,b]\to E^1$ and $g:[a,b]\to E^1$ are continuous then
\item [(i)] $\int\limits_a^b(\alpha f(x)+\beta g(x))dx=\alpha \int\limits_a^bf(x)dx+\beta\int\limits_a^bg(x)dx$ where $\alpha$ and $\beta$ are real numbers.
\item [(ii)] $\int\limits_a^bf(x)dx=\int\limits_a^cf(x)dx+\int\limits_c^bf(x)dx$ where $a<c<b$.
\item [(iii)]The function $F :[a,b]\to \mathbb{R_+}$ defined by $F(x)=D(f(x),g(x))$ is continuous on $[a, b]$ and\\ $$D\left(\int\limits_a^bf(x)dx,\int\limits_a^bg(x)dx \right)\leq \int\limits_a^bF(x)dx.$$
\item [(iv)] $\int\limits_a^xf(t)dt$ is a continuous function in $x\in[a,b]$.
\item [(v)] $\int\limits_a^b f(x)dx \preceq \int\limits_a^b g(x)dx $  whenever   $f(x)\preceq g(x)$ for all $x\in[a,b]$.
\end{theorem}
\begin{definition}
Suppose $f(x)$ is a fuzzy-number-valued function defined on the unbounded interval $[a, \infty)$. Then we define
\begin{eqnarray*}
\int_a^{\infty}f(x)dx=\lim_{t\to\infty}\int_a^{t}f(x)dx
\end{eqnarray*}
provided the limit on the right-hand side exists in $E^1$, in which case we say the integral converges and is equal to the value of limit. Otherwise, we say the integral diverges.
\end{definition}
\section{Main Results}
\begin{definition}
Let $f:[0,\infty)\rightarrow E^1$ be a continuous fuzzy-number-valued function  and $s(t)=\int\limits_0^{t}f(x)dx$. The Ces\`{a}ro means of $s(t)$ are defined by
\begin{eqnarray}\label{ortalama}
\sigma(t)=\frac{1}{t}\int_0^{t}s(u)du, \qquad t\in(0,\infty).
\end{eqnarray}
The integral
\begin{eqnarray}\label{integral}
\int_0^{\infty}f(x)dx
\end{eqnarray}
is said to be Ces\`{a}ro summable to a fuzzy number L if $\lim_{t\to \infty}\sigma(t)=L$.
The value of this limit is said to be the Ces\`{a}ro sum of the integral.
\end{definition}
\begin{theorem}\label{regular}
If the integral (\ref{integral}) converges to a fuzzy number L, then (\ref{ortalama}) also converges to L.
\end{theorem}
\begin{proof}
Let $$\lim\limits_{t\to \infty}s(t)=\int_0^{\infty}f(x)dx=L$$ for  some $L\in E^1$. Then given any $\varepsilon>0 $ there exists $t_0>0$ such that $D(s(t), L)<\frac{\varepsilon}{2}$ whenever $t\geq t_0$ and there exists $M>0$ such that $D(s(t), L)<M$ whenever $t< t_0$. So we have
\begin{eqnarray*}
D(\sigma(t),L)&=&D\left(\frac{1}{t}\int_0^ts(u)du, L\right)
\\&=&
D\left(\frac{1}{t}\int_0^ts(u)du,\frac{1}{t}\int_0^tLdu\right)
\\&=&
\frac{1}{t}D\left(\int_0^ts(u)du, \int_0^tLdu\right)
\\&\leq &
\frac{1}{t}\int_0^tD(s(u),L)du
\\&=&
\frac{1}{t}\int_0^{t_0}D(s(u),L)du+ \frac{1}{t}\int_{t_0}^tD(s(u),L)du
\\&\leq &
\frac{t_0M}{t}+\frac{\varepsilon}{2}\frac{(t-t_0)}{t}< \frac{t_0M}{t}+ \frac{\varepsilon}{2}
\end{eqnarray*}
Since $\lim\limits_{t\to\infty}\frac{t_0M}{t}=0$, there exists $t_1>0$ such that $\left|\frac{t_0M}{t}\right|<\frac{\varepsilon}{2}$ whenever $t\geq t_1$. So there exists $t_2=\max\{t_0,t_1\}$ such that $$D(\sigma(t),L)<\varepsilon$$ whenever $t\geq t_2$. This completes the proof.
\end{proof}
By the following example it can be easily seen that the converse statement of Theorem \ref{regular} is not true in general.
\begin{example}
Take the fuzzy-number-valued function  $f: [0,\infty) \to E^1$  such that
{\small\begin{eqnarray*}
(f(x))(t)=
\begin{cases}
(t-\cos x).(x+1)^2, & \quad \textrm{if} \quad  \cos x\leq t\leq \cos x+\frac{1}{(1+x)^2}, \\
2-(t-\cos x).(x+1)^2, &  \quad \textrm{if} \quad \cos x+\frac{1}{(1+x)^2}\leq t\leq \cos x+\frac{2}{(1+x)^2}, \\
0, & \qquad  \quad \textrm{otherwise}.
\end{cases}
\end{eqnarray*}}
Then $f$ is continuous and
{\small\begin{eqnarray*}
f^-_{\alpha}(x)=\cos x+\frac{\alpha}{(x+1)^2}\qquad\qquad &,&\qquad\qquad f^+_{\alpha}(x)=\cos x+\frac{2-\alpha}{(x+1)^2}\\
\int_{0}^{t}f_{\alpha}^-(x)dx=sint+\alpha\left(1-\frac{1}{t+1}\right)\quad &,& \quad \int_{0}^{t}f_{\alpha}^+(x)dx=sint+(2-\alpha)\left(1-\frac{1}{t+1}\right)
\end{eqnarray*}}
Obviously $\int\limits_{0}^{\infty}f(x)dx$ is divergent. To calculate Ces\`{a}ro mean, considering (\ref{ortalama}) we have
{\small\begin{eqnarray*}
\sigma_{\alpha}^-(t)&=&\frac{1}{t}\int_0^{t}s_{\alpha}^-(u)du=\frac{1}{t}\int_0^{t}\left(\int_0^{u}f_{\alpha}^-(x)dx\right)du=-\frac{\cos t}{t}+\frac{1}{t}+\alpha\left(1-\frac{\ln (t+1)}{t}\right)\\
\sigma_{\alpha}^+(t)&=&\frac{1}{t}\int_0^{t}s_{\alpha}^+(u)du=\frac{1}{t}\int_0^{t}\left(\int_0^{u}f_{\alpha}^+(x)dx\right)du=-\frac{\cos t}{t}+\frac{1}{t}+(2-\alpha)\left(1-\frac{\ln (t+1)}{t}\right).
\end{eqnarray*}}
So we get
\begin{eqnarray*}
\begin{rcases}
\lim_{t\to \infty}\sigma_{\alpha}^-(t)=\alpha \\
\lim_{t\to \infty}\sigma_{\alpha}^+(t)=2-\alpha\end{rcases}
\qquad [u]_{\alpha}=[\alpha, 2-\alpha] \qquad and \qquad \lim_{t\to \infty}D(\sigma(t),u)=0.
\end{eqnarray*}

Then $\int\limits_{0}^{\infty}f(x)dx$ is Ces\`{a}ro summable to fuzzy number $u$ such that
\begin{eqnarray*}
u(t)=
\begin{cases}
\ \ t &\quad \textrm{if} \quad  0\leq t\leq 1, \\
2-t &  \quad \textrm{if} \quad 1\leq t\leq 2, \\
\ \ 0 & \quad \quad  \ \  \textrm{otherwise}.
\end{cases}
\end{eqnarray*}
\end{example}
We need the following Lemma for the proofs of our main results.
\begin{lemma}\label{extra}
If $s$ be a continuous fuzzy-number-valued function then for every $\lambda>1$
\begin{eqnarray}\label{es1}
\frac{1}{\lambda t-t}\int_t^{\lambda t}s(x)dx + \frac{1}{\lambda-1}\sigma(t)=\sigma(\lambda t)+\frac{1}{\lambda-1}\sigma(\lambda t)
\end{eqnarray}
and for every $0<\ell<1$
\begin{eqnarray}\label{es2}
\frac{1}{t-\ell t}\int_{\ell t}^{ t}s(x)dx + \frac{\ell}{1-\ell}\sigma(\ell t)=\sigma( t)+\frac{\ell}{1-\ell}\sigma(t).
\end{eqnarray}
\end{lemma}
\begin{proof}
Let $s$ be a continuous fuzzy-number-valued function. Then for every $\lambda>1$ we have
\begin{eqnarray*}
\sigma(\lambda t)+\frac{1}{\lambda-1}\sigma(\lambda t)&=&\frac{\lambda}{\lambda-1}\sigma(\lambda t)
\\&=&
\frac{\lambda}{\lambda-1}\frac{1}{\lambda t}\int_0^{\lambda t}s(x)dx
\\&=&
\frac{1}{(\lambda-1)t}\left\{\int_0^{t}s(x)dx+\int_t^{\lambda t}s(x)dx\right\}
\\&=&
\frac{1}{\lambda-1}\sigma(t)+ \frac{1}{t(\lambda -1)}\int_t^{\lambda t}s(x)dx
\end{eqnarray*}
by Lemma \ref{added} and Theorem \ref{bessart}. On the other hand for every $0<\ell<1$, using Lemma \ref{added} and Theorem \ref{bessart} again,  we get
\begin{eqnarray*}
\sigma( t)+\frac{\ell}{1-\ell}\sigma(t)&=&\frac{1}{1-\ell}\sigma( t)
\\&=&
\frac{1}{1-\ell}\frac{1}{t}\int_0^{ t}s(x)dx
\\&=&
\frac{1}{1-\ell}\frac{1}{t}\left\{\int_0^{\ell t}s(x)dx+\int_{\ell t}^{t}s(x)dx\right\}
\\&=&
\frac{\ell}{1-\ell}\frac{1}{\ell t}\int_0^{\ell t}s(x)dx+ \frac{1}{t(1-\ell)}\int_{\ell t}^{t}s(x)dx
\\&=&
\frac{\ell}{1-\ell}\sigma(\ell t)+ \frac{1}{t-\ell t}\int_{\ell t}^{ t}s(x)dx.
\end{eqnarray*}
So equalities (\ref{es1}) and (\ref{es2}) are satisfied.
\end{proof}
As a result of Lemma \ref{extra} we conclude the following lemma.
\begin{lemma}\label{lemma}
If integral (\ref{integral}) is Ces\`{a}ro summable to a fuzzy number $L$, then for every $\lambda>1$
\begin{eqnarray}\label{L1}
\lim_{t\to\infty}\frac{1}{\lambda t-t}\int_t^{\lambda t}s(x)dx=L
\end{eqnarray}
and for every $0<\ell<1$
\begin{eqnarray}\label{L2}
\lim_{t\to\infty}\frac{1}{t-\ell t}\int_{\ell t}^{ t}s(x)dx =L.
\end{eqnarray}
\end{lemma}
Now we give Tauberian conditions under which convergence of the improper integral follows from Ces\`{a}ro summability.
\begin{theorem}\label{maintheorem}
Let fuzzy-number-valued function $f:[0,\infty)\rightarrow E^1$ be continuous. If integral (\ref{integral}) is Ces\`{a}ro summable to a fuzzy number L, then it converges to L if and only if for every $\varepsilon>0$ there exist $t_0\geq 0$ and $\lambda>1$ such that for $t> t_0$
\begin{eqnarray}\label{*}
\frac{1}{\lambda t-t}\int_t^{\lambda t}s(x)dx \succeq s(t)-\overline{\varepsilon}
\end{eqnarray}
and another $0<\ell<1$ such that
\begin{eqnarray}\label{**}
\frac{1}{t-\ell t}\int_{\ell t}^{ t}s(x)dx \preceq s(t)+\overline{\varepsilon}.
\end{eqnarray}
\end{theorem}
\begin{proof}{\it    Necessity}.
Let the integral (\ref{integral}) converge to $L$. Using inequality
\begin{eqnarray*}
D\left(\frac{1}{\lambda t-t}\int_t^{\lambda t}s(x)dx, s(t)\right)\leq D\left(\frac{1}{\lambda t-t}\int_t^{\lambda t}s(x)dx, L\right)+D(L,s(t)),
\end{eqnarray*}
if we consider the equality (\ref{L1}) in Lemma \ref{lemma} then for $\lambda>1$ we obtain
\begin{eqnarray*}
\lim_{t\to \infty}D\left(\frac{1}{\lambda t-t}\int_t^{\lambda t}s(x)dx, s(t)\right)=0.
\end{eqnarray*}
For $0<\ell<1$, validity of (\ref{**}) can also be obtained analogously by using the equality (\ref{L2}) of Lemma \ref{lemma}.

{\it    Sufficiency}. Assume that integral (\ref{integral}) is Ces\`{a}ro summable to $L$ and (\ref{*}), (\ref{**}) are satisfied. By (\ref{*}), there exist $t_1\geq 0$ and $\lambda>1$ such that for $t>t_1$
\begin{eqnarray*}
\frac{1}{\lambda t-t}\int_t^{\lambda t}s(x)dx \succeq s(t)-\frac{\overline{\varepsilon}}{3}\ \cdot
\end{eqnarray*}
Besides since
\begin{eqnarray*}
\lim\limits_{t\to \infty}D\left(\frac{1}{\lambda-1}\sigma(t), \frac{1}{\lambda-1}\sigma(\lambda t)\right)=0,
\end{eqnarray*}
there exists $t_2\geq 0$ such that for $t>t_2$
\begin{eqnarray*}
D\left(\frac{1}{\lambda-1}\sigma(t), \frac{1}{\lambda-1}\sigma(\lambda t)\right)\leq\frac{\varepsilon}{3}\cdot
\end{eqnarray*}
So by $(i)$ of Lemma \ref{epsilon} we get that
\begin{eqnarray*}
\frac{1}{\lambda-1}\sigma(t)- \frac{\overline{\varepsilon}}{3}\preceq \frac{1}{\lambda-1}\sigma(\lambda t)\preceq \frac{1}{\lambda-1}\sigma(t)+ \frac{\overline{\varepsilon}}{3}.
\end{eqnarray*}
Also, since $\lim\limits_{t\to \infty}\sigma(\lambda t)=L$, there exists $t_3\geq 0$ such that  $D(\sigma(\lambda t), L)\leq\frac{\varepsilon}{3}$ for $t> t_3$, meaning
\begin{eqnarray*}
L- \frac{\overline{\varepsilon}}{3}\preceq \sigma(\lambda t)\preceq L + \frac{\overline{\varepsilon}}{3}\cdot
\end{eqnarray*}
Then considering the equality (\ref{es1}) , there exists $t_4=\max\{t_1, t_2, t_3\}$ such that for $t> t_4$
\begin{eqnarray*}
s(t)-\frac{\overline{\varepsilon}}{3}+\frac{1}{\lambda-1}\sigma(t)\preceq L + \frac{\overline{\varepsilon}}{3}+\frac{1}{\lambda-1}\sigma(t)+ \frac{\overline{\varepsilon}}{3}\cdot
\end{eqnarray*}
So by $(v)$ of Lemma \ref{epsilon}, for $t> t_4$ we have
\begin{eqnarray}\label{epsilon1}
s(t)\preceq L + \overline{\varepsilon} .
\end{eqnarray}
On the other hand, if we consider the condition (\ref{**}),  equality (\ref{es2}), Lemma \ref{epsilon} and proceed in a similar way as that above, we get that there exists a $t_4^*\geq 0$ such that for $t> t_4^*$
\begin{eqnarray}\label{epsilon2}
s(t)\succeq L - \overline{\varepsilon}.
\end{eqnarray}
Then combining inequalities (\ref{epsilon1}) and (\ref{epsilon2}), we obtain
\begin{eqnarray*}
L - \overline{\varepsilon}\preceq s(t)\preceq L + \overline{\varepsilon}
\end{eqnarray*}
whenever $t>\max\{t_4,t_4^*\}$ and this completes the proof.
\end{proof}
\begin{definition}
A fuzzy-number-valued function $s(x)$ is said to be slowly decreasing if for every $\varepsilon>0$ there exist $t_0\geq 0$ and $\lambda >1$ such that
\begin{eqnarray*}
s(x)\succeq s(t)- \overline{\varepsilon}
\end{eqnarray*}
whenever $t_0<t<x\leq \lambda t.$
\end{definition}
\begin{remark}
Fuzzy-number-valued function $s(x)$ is slowly decreasing if and only if the family of real valued functions $\{s^-_{\alpha}(x)\mid \alpha\in[0,1]\}$ and $\{s^+_{\alpha}(x)\mid \alpha\in[0,1]\}$ are equi-slowly decreasing i.e.  $\forall\varepsilon>0$ there exist $t_0\geq 0$ and $\lambda>1$ such that for all $\alpha\in[0,1]$ $$s_{\alpha}^-(x)-s_{\alpha}^-(t)\geq -\varepsilon \quad \textrm{and}\quad s_{\alpha}^+(x)-s_{\alpha}^+(t)\geq -\varepsilon \qquad \textrm{whenever} \quad t_0< t<x\leq \lambda t.$$
\end{remark}
\begin{lemma}
If the fuzzy-number-valued  function $s(x)$ is slowly decreasing, then for every $\varepsilon>0$ there exist $t_0\geq0$ and  $0<\lambda<1$ such that for every $t>t_0$
\begin{eqnarray}
s(t)\succeq s(x)-\overline{\varepsilon} \qquad \textrm{whenever}\qquad  \lambda t< x\leq t.
\end{eqnarray}
\end{lemma}
\begin{proof}
The proof of the lemma is done by contradiction method. Assume that the fuzzy-number-valued function $s(x)$ is slowly decreasing and there exists $\varepsilon_0>0$ such that for all $0<\lambda<1$ and $t_0\geq0$ there exist real numbers $x$ and $t>t_0$ for which
\begin{eqnarray}
s(t)\nsucceq s(x)-\overline{\varepsilon}_0\qquad \textrm{whenever}\qquad  \lambda t< x\leq t.
\end{eqnarray}
Therefore, there exists $\alpha_{\scriptscriptstyle 0}\in [0,1]$ such that
\begin{eqnarray}\label{cases}
s^-_{\alpha_{ 0}}(t)<s^-_{\alpha_{ 0}}(x)-\varepsilon_0 \quad or\quad s^+_{\alpha_0}(t)<s^+_{\alpha_0}(x)-\varepsilon_0.
\end{eqnarray}
At this point we recall the reformulated condition of M$\acute{o}$ricz \cite{moric1} for a slowly decreasing real valued function $f$ such that
\begin{eqnarray}\label{soldan}
\lim_{\lambda\rightarrow 1^{-}}\liminf_{t\rightarrow\infty}\min_{\lambda t\leq x\leq t}[f(t)-f(x)]\geq0.
\end{eqnarray}
No matter which case we choose in (\ref{cases}), one of the real valued functions $s^-_{\alpha_{ 0}}(t)$ and $s^+_{\alpha_0}(t)$ does not satisfy the condition (\ref{soldan}). So at least one of them is not slowly decreasing which contradicts the hypothesis that fuzzy-number-valued function $s(x)$ is slowly decreasing.
\end{proof}

It is clear that if function $s$ is slowly decreasing then conditions (\ref{*}) and (\ref{**}) are satisfied by $(i)$ and $(v)$ of Theorem \ref{bessart}.  So next corollary immediately follows:
\begin{corollary}\label{maintheorem2}
If $f$ is a continuous fuzzy-number-valued function such that integral (\ref{integral}) is Ces\`{a}ro summable to a fuzzy number $L$ and its integral function $s(t)$ is slowly decreasing, then the integral (\ref{integral}) converges to $L$.
\end{corollary}
\begin{theorem}\label{Landau}
Let $f$ be a continuous fuzzy-number-valued function on $[0,\infty)$. If there exist negative constant fuzzy number $u$ and a real number $x_0\geq 0$ such that
\begin{eqnarray}\label{one}
xf(x)\succeq u \quad \textrm{for}\quad x>x_0,
\end{eqnarray}
then fuzzy-number-valued function $s(t)=\int\limits_{0}^{t}f(x)dx$ is slowly decreasing.
\end{theorem}
\begin{proof}
Let $ xf(x)\succeq u$ be satisfied under the given conditions on $u$ and $x_0$ in the theorem. Then for $x>x_0$ we have
\begin{eqnarray*}
xf_{\alpha}^-(x)\geq u^-_{\alpha}\geq u^-_{0}\qquad ,\qquad xf_{\alpha}^+(x)\geq u^+_{\alpha}\geq u^+_{1}\geq u^-_{0}.
\end{eqnarray*}
For the sake of simplicity let take $u^-_{0}=-H$ where $H> 0$. Then
\begin{eqnarray*}
xf_{\alpha}^-(x)\geq -H \Rightarrow f_{\alpha}^-(x)\geq -\frac{H}{x}\quad ,\quad
xf_{\alpha}^+(x)\geq -H \Rightarrow f_{\alpha}^+(x)\geq -\frac{H}{x}
\end{eqnarray*}
are satisfied. Then for $x_0< t<x\leq \lambda t$ when $\lambda>1$, we have
\begin{eqnarray*}
s_{\alpha}^-(x)-s_{\alpha}^-(t)=\int_t^xf_{\alpha}^-(u)du\geq -H\int_t^x\frac{du}{u}=-H\ln\frac{x}{t}\geq -H\ln \lambda
\end{eqnarray*}
and
\begin{eqnarray*}
s_{\alpha}^+(x)-s_{\alpha}^+(t)=\int_t^xf_{\alpha}^+(u)du\geq -H\int_t^x\frac{du}{u}=-H\ln\frac{x}{t}\geq -H\ln \lambda.
\end{eqnarray*}
Choosing $\lambda =e^{\varepsilon/H}$, we get the inequalities
\begin{eqnarray*}
s_{\alpha}^-(x)\geq s_{\alpha}^-(t)-\varepsilon\qquad ,\qquad
s_{\alpha}^+(x)\geq s_{\alpha}^+(t)-\varepsilon
\end{eqnarray*}
and then $s(x)\succeq s(t)-\overline{\varepsilon}$ holds whenever $x_0< t<x\leq \lambda t$.
\end{proof}
\begin{example}
Let the fuzzy-number-valued function $f :[0,\infty)\to E^1$ be given as
\begin{eqnarray*}
(f(x))(t)=
\begin{cases}
\quad\frac{t}{2-\sin x}\quad, &\quad \textrm{if} \quad   0\leq t\leq 2-\sin x , \\[5pt]
2-\frac{t}{2-\sin x}\ , & \quad  \textrm{if} \quad  2-\sin x\leq t\leq 2(2-\sin x).
\end{cases}
\end{eqnarray*}
Then
\begin{eqnarray*}
f_{\alpha}^-(x)=(2-\sin x)\alpha \qquad \qquad , \qquad \quad f_{\alpha}^+(x)=(2-\sin x)(2-\alpha).
\end{eqnarray*}
Since $f_{\alpha}^{\pm}(x)\geq 0$ holds for each $\alpha \in [0,1]$ and $x>0$, we have
\begin{eqnarray*}
xf_{\alpha}^-(x)\geq 0 \qquad , \qquad xf_{\alpha}^+(x)\geq 0 \quad
\end{eqnarray*}
which means that $xf(x)\succeq \overline{0}$. So $s(t)$ is slowly decreasing.
\end{example}
As a result of Theorem \ref{Landau} the following one-sided Tauberian condition is obtained.
\begin{corollary}
 If $f$ is a continuous fuzzy-number-valued function  such that integral (\ref{integral}) is Ces\`{a}ro summable to a fuzzy number L and condition (\ref{one}) is satisfied, then the integral (\ref{integral}) converges to $L$.
\end{corollary}
We note that one may extend Ces\`{a}ro summability method to continuous fuzzy-number-valued functions and give analogs of Theorem \ref{regular}--\ref{maintheorem}, Corollary \ref{maintheorem2} for Ces\`{a}ro summability of fuzzy-number-valued functions. The proofs are done identically by replacing integral function $s$ with general continuous fuzzy-number-valued function in corresponding proofs and hence omitted.
\begin{definition}
A continuous fuzzy-number-valued function $f:[0,\infty)\rightarrow E^1$ is said to be Ces\`{a}ro summable to a fuzzy number $L$ if
\begin{eqnarray*}
\lim_{t\to\infty}\frac{1}{t}\int_0^{t}f(x)dx=L.
\end{eqnarray*}
\end{definition}
\begin{theorem}
Let  $f$ be a continuous fuzzy-number-valued function. If \ $\lim_{t\to\infty}f(t)=L$, then $f$ is Ces\`{a}ro summable to fuzzy number $L$.
\end{theorem}
\begin{theorem}
If a continuous fuzzy-number-valued function $f$ is Ces\`{a}ro summable to a fuzzy number $L$, then $\lim_{t\to\infty}f(t)=L$ if and only if for every $\varepsilon>0$ there exist $t_0\geq 0$ and $\lambda>1$ such that for $t> t_0$
\begin{eqnarray*}
\frac{1}{\lambda t-t}\int_t^{\lambda t}f(x)dx \succeq f(t)-\overline{\varepsilon}
\end{eqnarray*}
and another $0<\ell<1$ such that
\begin{eqnarray*}
\frac{1}{t-\ell t}\int_{\ell t}^{ t}f(x)dx \preceq f(t)+\overline{\varepsilon}.
\end{eqnarray*}
\begin{theorem}
If a continuous fuzzy-number-valued function $f$ is Ces\`{a}ro summable to a fuzzy number $L$ and $f$ is slowly decreasing, then $\lim_{t\to\infty}f(t)=L$.
\end{theorem}

\end{theorem}

\end{document}